\documentclass[12pt]{article}
\usepackage{url}
\usepackage{amsfonts}
\usepackage{amsmath,amssymb,color}
\usepackage{graphicx}
\usepackage{enumitem}

\usepackage{subcaption}
\usepackage{multirow}
\usepackage{bm}
\usepackage{apacite}
\usepackage{verbatim}
\usepackage{float}
\usepackage[toc,page]{appendix}
\usepackage[onehalfspacing]{setspace}
\usepackage{setspace}
\usepackage{natbib}

\newcommand{\ttt}{\boldsymbol \theta}
\newcommand{\TT}{\Theta}

\newcommand{\bb}{\mbox{$\mathbf b$}}

\newcommand{\dd}{\mbox{$\mathbf d$}}

\newcommand{\EE}{\mathbf 1}

\newcommand{\G}{\mbox{$\mathbb G$}}

\newcommand{\II}{\mbox{$\mathbf I$}}

\newcommand{\aaaa}{\mathbf a}

\newcommand{\hh}{\mathbf h}

\newcommand{\cc}{\mathbf c}
\newcommand{\ff}{\mathbf f}
\newcommand{\gggg}{\mathbf g}
\newcommand{\HH}{\mathbf H}

\newcommand{\R}{\mathbb R}

\newcommand{\PP}{\mathbf P}
\newcommand{\PPP}{\mathcal P}

\newcommand{\vv}{\mathbf v}

\newcommand{\x}{\mathbf x}
\newcommand{\xx}{\mathbf x}
\newcommand{\XX}{\mathbf X}

\newcommand{\ZZ}{\mathbf Z}

\newcommand{\argmax}{\operatornamewithlimits{arg\,max}}

\newtheorem{theorem}{Theorem}

\newtheorem{definition}{Definition}

\newtheorem{preexample}{Example}
\newenvironment{example}%
  {\begin{preexample}\upshape}{\end{preexample}}

\newtheorem{lemma}{Lemma}

\newtheorem{preremark}{Remark}
\newenvironment{remark}%
  {\begin{preremark}\upshape}{\end{preremark}}

\title{A Note on Likelihood Ratio Tests for Models with Latent Variables}

\author{Yunxiao Chen, \\
 London School of Economics and Political Science\\
 Irini Moustaki,\\
 London School of Economics and Political Science\\
 Haoran Zhang, Fudan University}

 \date{}

\begin{document}
\maketitle

\doublespacing

\begin{abstract}
The likelihood ratio test (LRT) is widely used for comparing the relative fit of nested latent variable models.
Following Wilks' theorem, the LRT is conducted by comparing the LRT statistic with its asymptotic distribution under the restricted model, a $\chi^2$-distribution with degrees of freedom equal to
the difference in the number of free parameters between the two nested models under comparison.
For models with latent variables such as factor analysis, structural equation models  and random effects models, however, it is often found that the $\chi^2$ approximation does not hold.
In this note, we show how the regularity conditions of Wilks' theorem may be violated using three examples of models with latent variables.
In addition, a more general theory for LRT is given that provides the correct asymptotic theory for these LRTs. This general theory was first established in \cite{chernoff1954distribution} and discussed in both \cite{van2000asymptotic} and \cite{drton2009likelihood}, but it does not seem to have received enough attention. We illustrate this general theory with the three examples.
\end{abstract}	

\section{Introduction}
 \subsection{Literature on Likelihood Ratio Test}

The likelihood ratio test (LRT) is one of the most popular methods for comparing nested models.
When comparing two nested models that satisfy certain regularity conditions,
the $p$-value of an LRT is obtained by comparing the LRT statistic with a $\chi^2$-distribution with degrees of freedom equal to the difference in the number of free parameters between the two nested models. This reference distribution is suggested by the asymptotic theory of LRT that is known as Wilks' theorem \citep{wilks1938large}.

However, for the statistical inference of models with latent variables (e.g. factor analysis, item factor analysis for categorical data, structural equation models, random effects models, finite mixture models), it is often found that the $\chi^2$ approximation suggested by Wilks' theorem does not hold.
There are various published
studies showing that the LRT is not valid under certain violations/conditions (e.g. small sample size, wrong model under the alternative hypothesis, large number of items, non-normally distributed variables, unique variances equal to zero, lack of identifiability),  leading to over-factoring and over rejections; see e.g.
\cite{hakstian.ea:82}, \cite{liu2003asymptotics},
\cite{hayashi2007likelihood}, \cite{asparouhov.muthen:2009}, \cite{wu2016identification},
\cite{deng2018structural}, \cite{shi2018revisiting}, \cite{yang2018performance}
and \cite{auerswald.moshagen:2019}.  There is also a significant amount of literature on the effect of testing at the boundary of parameter space
that arise when testing the significance of variance components in random effects models as well as in structural equation models (SEM) with linear or nonlinear constraints \cite[see][]{stram.lee:1994,stram.lee:1995,dominicus.ea:2006,savalei.kolenikov:2008,davis2009analysis,wu2013likelihood, du2020testing}.

Theoretical investigations have shown that certain regularity conditions of Wilks' theorem are not always satisfied when comparing nested models with latent variables. \cite{takane2003ea} and \cite{hayashi2007likelihood} were among the ones who pointed out that models for which one needs to select dimensionality (e.g. principal component analysis, latent class, factor models)  have points of irregularity in their parameter space that in some cases invalidate the use of LRT. Specifically, such
issues arise in factor analysis when  comparing models with different number of factors rather than  comparing a factor model against the saturated model. The LRT for comparing a $q$-factor model against the saturated model does follow a $\chi^2$-distribution under mild conditions. However, for nested models with different number of factors ($q$-factor model is the correct one against the one with
$(q+k)$-factors), the LRT is likely not  $\chi^2$-distributed due to violation of one or more of the regularity conditions.
This is inline with the two basic assumptions required by the asymptotic theory for  factor analysis and SEM: the identifiability of the parameter vector and non-singularity of the information matrix \cite[see][and references therein]{shapiro1986}.
More specifically, \cite{hayashi2007likelihood} focus on exploratory factor analysis and on the problem that arises when the number of factors exceeds the true number of factors that might lead to rank deficiency and nonidentifiability of model parameters. That corresponds to the violations of the two regularity conditions. Those findings go back to   \cite{geweke1980interpreting} and \cite{amemiya.anderson:1990}. More specifically, \cite{geweke1980interpreting} studied the behaviour of the LRT in small samples and concluded that when the regularity conditions from Wilks' theorem are not satisfied the asymptotic theory seems to be misleading in all sample sizes considered.


\subsection{Our Contributions}

The contribution of this note is two-folds. First, we provide a discussion about situations under which
Wilks' theorem for LRT may fail. Via three examples, we provide a relatively more complete picture about this issue  in models with latent variables.
Second, we introduce a unified asymptotic theory for LRT that covers Wilks' theorem as a special case and provides the correct asymptotic reference distribution for LRT when Wilks' theorem fails. This unified theory does not seem to have received enough attention in psychometrics, even though it has been established in statistics for long  \citep{chernoff1954distribution, van2000asymptotic, drton2009likelihood}. In this note, we provide a tutorial on this theory, by presenting the theorems in a more accessible way and providing illustrative examples.

%

\subsection{Examples}
To further illustrate the issue with the classical theory for LRT, we provide three examples. These examples suggest that the $\chi^2$ approximation can perform poorly and give $p$-values that can be either more conservative or more liberal.

\begin{example}{\bf (Exploratory factor analysis).}\label{eg1}
Consider a dimensionality test in exploratory factor analysis (EFA). For ease of exposition, we consider two hypothesis testing problems, (a) testing a one-factor model against a two-factor model, and (b) testing a one-factor model against a saturated multivariate normal model with an unrestricted covariance matrix. Similar examples have been considered in \cite{hayashi2007likelihood} where similar phenomena have been studied.

\paragraph{1(a).} Suppose that we have $J$ mean-centered continuous indicators, $\XX = (X_1, ..., X_J)^\top$, which follow a $J$-variate normal distribution $N(\mathbf 0, \boldsymbol\Sigma)$. The one-factor model parameterizes $\boldsymbol\Sigma$ as
$$\boldsymbol\Sigma = \aaaa_1\aaaa_1^\top + \boldsymbol\Delta,$$
where $\aaaa_1 = (a_{11}, ..., a_{J1})^\top$ contains the loading parameters and $\boldsymbol\Delta = diag(\delta_1, ..., \delta_J)$ is diagonal matrix
with a diagonal entries $\delta_1$, ..., $\delta_J$. Here, $\boldsymbol\Delta$ is the covariance matrix for the unique factors. Similarly, the two-factor model
parameterizes $\boldsymbol\Sigma$ as
$$\boldsymbol\Sigma = \aaaa_1\aaaa_1^\top + \aaaa_2\aaaa_2^\top + \boldsymbol\Delta,$$
where  $\aaaa_2 = (a_{12}, ..., a_{J2})^\top$  contains the loading parameters for the second factor and we set $a_{12} = 0$ to ensure model identifiability. Obviously, the one-factor model is nested within the two-factor model. The comparison between these two models is equivalent to test
$$H_0: \aaaa_2 = \mathbf 0 \mbox{~versus~} H_a: \aaaa_2 \neq \mathbf 0.$$
If Wilks' theorem holds, then under $H_0$ the LRT statistic should asymptotically follow a $\chi^2$-distribution with $J-1$ degrees of freedom.

We now provide a simulated example. Data are generated from a one-factor model, with $J = 6$ indicators and $N=5000$ observations. The true parameter values are given in Table~\ref{tab:eg1}. We generate 5000 independent datasets. For each dataset, we compute the LRT for comparing the one- and two-factor models. Results are presented in panel (a) of Figure~\ref{fig:fig1}.
The black solid line shows the empirical Cumulative Distribution Function (CDF) of the LRT statistic, and the red dashed line shows the CDF of the $\chi^2$ distribution suggested by Wilks' Theorem. A substantial discrepancy can be observed between the two CDFs.  Specifically, the $\chi^2$ CDF tends to stochastically dominate the empirical CDF, implying that p-values based on this $\chi^2$ distribution tend to be more liberal.
In fact, if we reject $H_0$ at 5\% significance level based on these p-values, the actual type I error is 10.8\%.
These results suggest the  failure of Wilks' theorem in this example.



\begin{table}
  \centering
  \begin{tabular}{cccccc}
    \hline
    $a_{11}$ & $a_{21}$& $a_{31}$ & $a_{41}$& $a_{51}$ & $a_{61}$\\
    \hline
    1.17& 1.87& 1.42& 1.71& 1.23& 1.78 \\
   \hline
   \hline
    $\delta_{1}$ & $\delta_{2}$& $\delta_{3}$ & $\delta_{4}$& $\delta_{5}$ & $\delta_{6}$\\
    \hline
    1.38& 0.85& 1.46& 0.78& 1.24& 0.60\\
    \hline
  \end{tabular}
  \caption{The values of the true parameters for the simulations in Example~\ref{eg1}.}\label{tab:eg1}
\end{table}

\begin{figure}
  \centering
  \minipage{0.48\textwidth}
  \includegraphics[width=\linewidth]{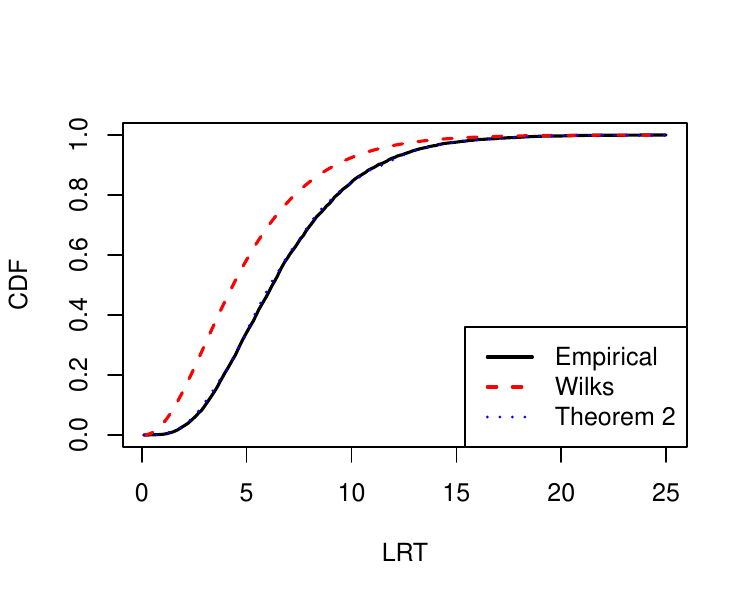}
  \subcaption{}
\endminipage\hfill
\minipage{0.48\textwidth}
  \includegraphics[width=\linewidth]{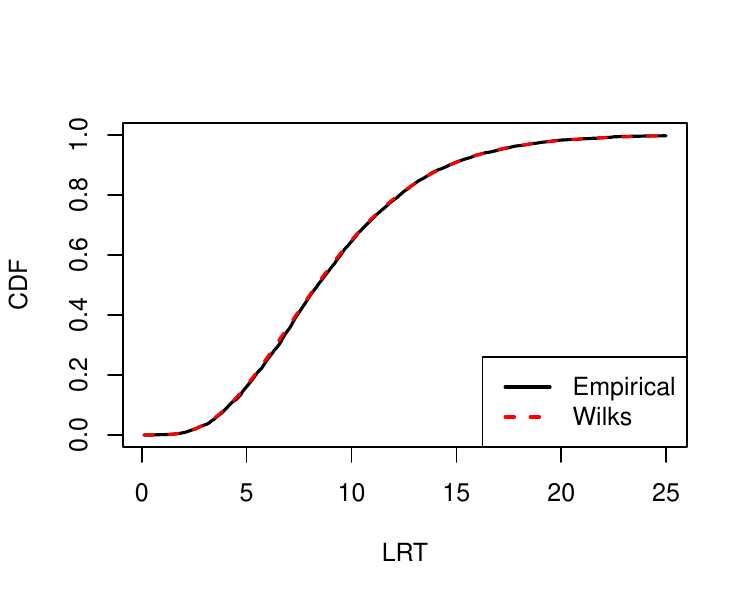}
    \subcaption{}
\endminipage\hfill
  \caption{Panel (a) shows the results of Example 1(a). The black solid line shows the empirical CDF of the LRT statistic, based on 5000 independent simulations. The red dashed line shows the CDF of the $\chi^2$ distribution with 5 degrees of freedom as suggested by Wilks' theorem. The blue dotted line shows the CDF of the reference distribution suggested by Theorem~\ref{thm:LRT2}.  Panel (b) shows the results of Example 1(b). The black solid line shows the empirical CDF of the LRT statistic, and the red dashed line shows the CDF of the $\chi^2$ distribution with 9 degrees of freedom as suggested by Wilks' theorem.
  }\label{fig:fig1}
\end{figure}

\paragraph{1(b).} When testing the one-factor model against the saturated model, the LRT statistic is asymptotically $\chi^2$ if Wilks' theorem holds.  The degrees of freedom of the $\chi^2$ distribution is $J(J+1)/2 - 2J$, where $J(J+1)/2$ is the number of free parameters in an unrestricted covariance matrix $\boldsymbol\Sigma$ and $2J$ is the number of parameters in the one-factor model.
In panel (b) of Figure~\ref{fig:fig1},
the black solid line shows the empirical CDF of the LRT statistic based on 5000 independent simulations, and the red dashed line shows the CDF of the $\chi^2$-distribution with 9 degrees of freedom.
As we can see, the two curves almost overlap with each other, suggesting that Wilks' theorem holds here.


\end{example}

\begin{example}{\bf (Exploratory item factor analysis).}\label{eg2}
 We further give an example of exploratory item factor analysis (IFA) for binary data, in which similar phenomena as those in Example~\ref{eg1} are observed. Again, we consider two hypothesis
testing problems, (a) testing a one-factor model against a two-factor model, and (b) testing a one-factor model against a saturated multinomial model for a binary random vector.

\paragraph{2(a).} Suppose that we have a $J$-dimensional response vector, $\XX = (X_1, ..., X_J)^\top$, where all the entries are binary valued, i.e., $X_j \in \{0, 1\}$. It follows a categorical distribution, satisfying
$$P(\XX = \xx) = \pi_{\xx}, \xx \in \{0,1\}^J,$$
where $\pi_\xx \geq 0$ and $\sum_{\xx \in \{0, 1\}^J} \pi_{\xx} = 1$.

The exploratory two-factor IFA model parameterizes $\pi_\xx$ by
$$\pi_{\xx} = \int\int \prod_{j=1}^J \frac{\exp(x_j(d_j + a_{j1}\xi_1 + a_{j2}\xi_2))}{1+\exp(d_j + a_{j1}\xi_1 + a_{j2}\xi_2)} \phi(\xi_1)\phi(\xi_2)d\xi_1d\xi_2,$$
where $\phi(\cdot)$ is the probability density function of a standard normal distribution.
This model is also known as a multidimensional two-parameter logistic (M2PL) model \citep{reckase2009multidimensional}.
Here, $a_{jk}$s are known as the discrimination parameters and $d_j$s are known as the easiness parameters. We denote $\aaaa_1 = (a_{11},...,a_{J1})^\top$ and $\aaaa_2 = (a_{12},...,a_{J2})^\top.$
For model identifiability, we set $a_{12} = 0$. When $a_{j2} = 0$, $j=2, ..., J$, then the two-factor model degenerates to the one-factor model.
Similar to Example 1(a), if Wilks' theorem holds, the LRT statistic should asymptotically follow a $\chi^2$-distribution with $J-1$ degrees of freedom.

Simulation results suggest the failure of this $\chi^2$ approximation. In Figure~\ref{fig:fig2}, we provide plots similar to those in Figure~\ref{fig:fig1}, based on $5000$ datasets simulated from a one-factor IFA model with sample size $N = 5000$ and $J = 6$. The true parameters of this IFA model are given  in Table~\ref{tab:eg2}.
The result is shown in panel (a) of Figure~\ref{fig:fig2}, where a similar pattern is observed as that in panel (a) of Figure~\ref{fig:fig1} for Example 1(a).



\begin{table}
  \centering
  \begin{tabular}{cccccc}
     \hline
    $d_{1}$ & $d_{2}$& $d_{3}$ & $d_{4}$& $d_{5}$ & $d_{6}$\\
    \hline
     -0.23  & -0.12& 0.07& 0.31 &-0.29 & 0.19\\
    \hline
    \hline
    $a_{11}$ & $a_{21}$& $a_{31}$ & $a_{41}$& $a_{51}$ & $a_{61}$\\
    \hline
     0.83& 1.22& 0.96& 0.91& 1.02& 1.25 \\
   \hline

  \end{tabular}
  \caption{The values of the true parameters for the simulations in Example~\ref{eg2}.}\label{tab:eg2}
\end{table}

\begin{figure}
  \centering
  \minipage{0.48\textwidth}
  \includegraphics[width=\linewidth]{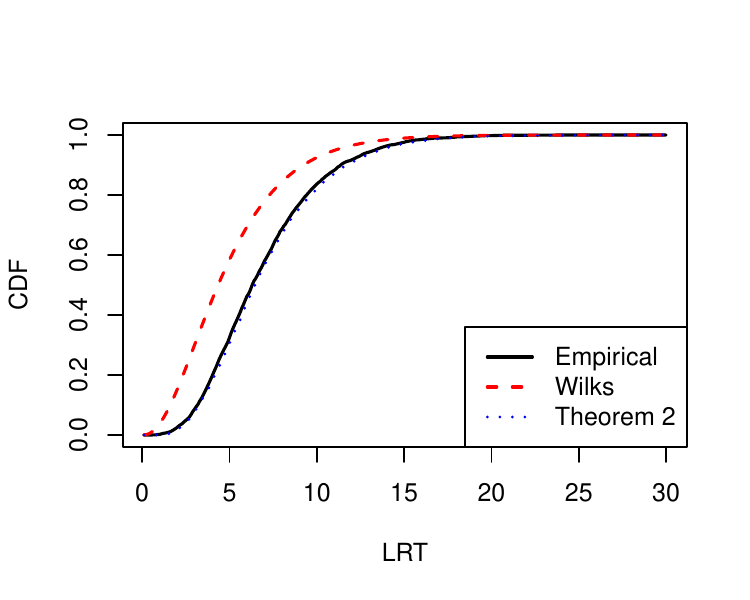}
      \subcaption{}
\endminipage\hfill
\minipage{0.48\textwidth}
  \includegraphics[width=\linewidth]{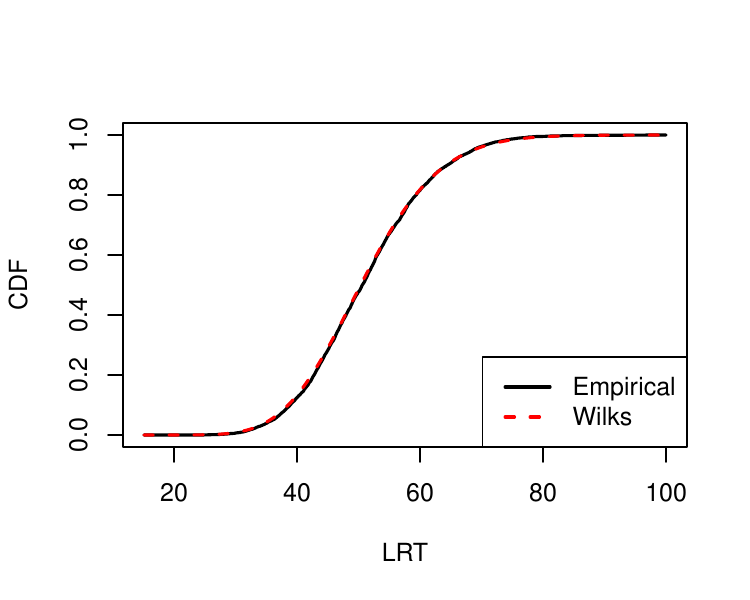}
      \subcaption{}
\endminipage
  \caption{Panel (a) shows the results of Example 2(a). The black solid line shows the empirical CDF of the LRT statistic, based on 5000 independent simulations. The red dashed line shows the CDF of the $\chi^2$-distribution with 5 degrees of freedom as suggested by Wilks' theorem. The blue dotted line shows the CDF of the reference distribution suggested by Theorem~\ref{thm:LRT2}.  Panel (b) shows the results of Example 2(b). The black solid line shows the empirical CDF of the LRT statistic, and the red dashed line shows the CDF of the $\chi^2$-distribution with 51 degrees of freedom as suggested by Wilks' theorem.
  }\label{fig:fig2}
\end{figure}

\paragraph{2(b).}
When testing the one-factor IFA model against the saturated model, the LRT statistic is asymptotically $\chi^2$ if Wilks' theorem holds, for which the degree of freedom  is $2^J-1 - 2J$. Here, $2^J-1$ is the number of free parameters in
the saturated model, and $2J$ is the number of parameters in the one-factor IFA model.
The result is given in panel (b) of Figure~\ref{fig:fig2}.
Similar to Example 1(b), the empirical CDF and the CDF implied by Wilks' theorem are very close to each other, suggesting that Wilks' theorem holds here.




\end{example}

\begin{example}{\bf (Random effects model).}\label{eg3}
Our third example considers a random intercept model.  Consider two-level data with individuals at level 1 nested within groups at level 2. Let $X_{ij}$ be data from the $j$th individual from the $i$th group, where $i = 1, ..., N$ and $j = 1, ..., J$. For simplicity, we assume all the groups have the same number of individuals. Assume the following random effects model,
$$X_{ij} = \beta_0 + \mu_i + \epsilon_{ij}, $$
where $\beta_0$ is the overall mean across all the groups, $\mu_i \sim N(0, \sigma_1^2)$ characterizes the difference between the mean for group $i$ and the overall mean, and  $\epsilon_{ij} \sim N(0, \sigma_2^2)$ is the individual level residual.

To test for between group variability under this model is equivalent to test
$$H_0: \sigma_1^2 = 0 \mbox{~against~} H_a: \sigma_1^2 > 0.$$
If Wilks' theorem holds, then the LRT statistic should follow a $\chi^2$ distribution with one degree of freedom.
We conduct a simulation study and show the results in Figure~\ref{fig:fig3}.
In this figure,
the black solid line shows the empirical CDF of the LRT statistic,
based on 5000 independent simulations from the null model with $N = 200$, $J= 20$, $\beta_0 = 0$, and $\sigma^2_2 = 1$.
The red dashed line shows the CDF of the $\chi^2$ distribution with one degree of freedom. As we can see, the two CDFs are not close to each other, and the empirical CDF tends to stochastically dominate the theoretical CDF  suggested by Wilks' theorem. It suggests the failure of Wilks' theorem in this example.


This kind of phenomenon has been observed when the null model lies on the boundary of the parameter space, due to which the regularity conditions of Wilks' theorem do not hold. The LRT statistic has been shown to often follow a mixture of $\chi^2$-distribution asymptotically \citep[e.g.,][]{shapiro1985asymptotic,self.liang:1987}, instead of a $\chi^2$-distribution. As it will be shown in Section~\ref{sec:theory}, such a mixture of $\chi^2$ distribution can be derived from  a general theory for LRT.

\begin{figure}
  \centering
  \includegraphics[scale = 0.6]{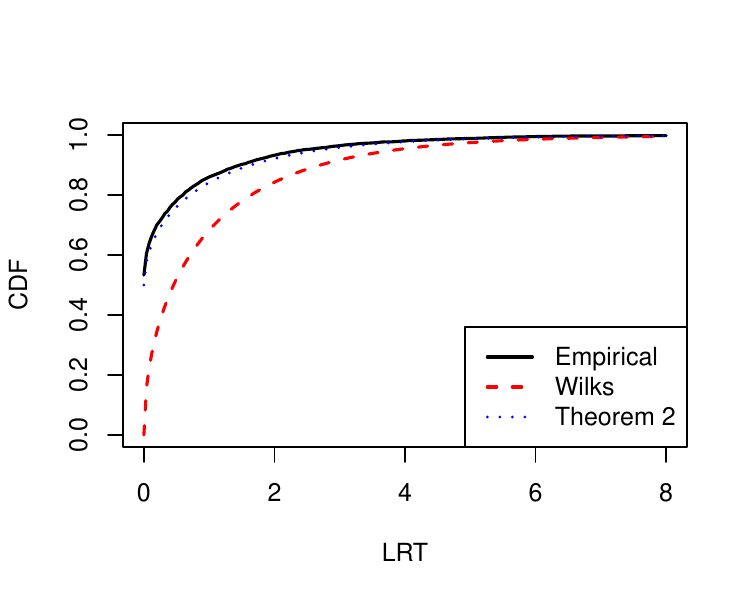}
    \caption{The black solid line shows the empirical CDF of the LRT statistic, based on 5000 independent simulations. The red dashed line shows the CDF of the $\chi^2$-distribution with one-degree of freedom as suggested by Wilks' theorem. The blue dotted line shows the CDF of the mixture of $\chi^2$-distribution suggested by Theorem~\ref{thm:LRT2}.}\label{fig:fig3}
\end{figure}
%
%
\end{example}

We now explain why Wilks' theorem does not hold in Examples 1(a), 2(a), and 3. We define some generic notations. Suppose that we have i.i.d. observations $\XX_1$, ..., $\XX_N$, from a parametric model $\mathcal P_{\Theta} = \{P_{\boldsymbol\theta}: \boldsymbol\theta  \in \Theta \subset \mathbb R^k\}$,
where $\XX_i = (X_{i1}, ..., X_{iJ})^\top.$
We assume that the distributions in $\mathcal P_{\Theta}$ are dominated by
a common $\sigma$-finite measure $\nu$ with respect to which they have probability density functions $p_{\boldsymbol\theta}: \mathbb R^J \rightarrow [0,\infty)$. Let $\Theta_0 \subset \Theta$ be a submodel and we are interested in testing
$$H_0: \ttt \in \Theta_0 \mbox{~versus~} H_a:  \ttt \in \Theta\setminus \Theta_0.$$
Let $p_{\boldsymbol\theta^* }$ be the true model for the observations, where $\boldsymbol\theta^* \in \Theta_0$.

The likelihood function is defined as
$$l_N(\ttt) = \sum_{i=1}^N \log p_{\boldsymbol\theta}(\XX_i),$$
and the LRT statistic is defined as
$$\lambda_N = 2\left(\sup_{\ttt\in \Theta} l_N(\ttt) - \sup_{\ttt\in \Theta_0} l_N(\ttt)\right).$$
Under suitable regularity conditions, Wilks' theorem suggests that the LRT statistic $\lambda_N$ is asymptotically $\chi^2$.

Wilks' theorem for LRT requires several regularity conditions; see e.g., Theorem 12.4.2, \cite{lehmann2006testing}. Among these conditions, there are two conditions that the previous examples do not satisfy.
First, it is required that $\ttt^*$ is an interior point of $\Theta$. This condition is not satisfied for Example 3, when $\Theta$ is taken to be $\{(\beta_0, \sigma_1^2, \sigma_2^2): \beta_0 \in \mathbb R, \sigma_1^2 \in [0, \infty), \sigma_2^2 \in [0, \infty)\}$,
as the null model lies on the boundary of the parameter space. Second, it is required that the expected Fisher information matrix at $\ttt^*$, $I(\boldsymbol\theta^*) = E_{\ttt^*}[\nabla l_{N}(\ttt^*)\nabla l_{N}(\ttt^*)^\top]/N$
is strictly positive-definite.
As we summarize in Lemma~\ref{lemma:lemma1},
this condition is not satisfied in Examples 1(a) and 2(a), when $\Theta$ is taken to be the parameter space of the corresponding two-factor model. However, interestingly, when comparing the one-factor model with the saturated model, the Fisher information matrix is strictly positive-definite in Examples 1(b) and 2(b), for both simulated examples.




\begin{lemma}\label{lemma:lemma1}

(1) For the two-factor model given in Example 1(a),
choose the parameter space to be $$
\TT  = \left\{ (\delta_1,...,\delta_J, a_{11},...,a_{J1},a_{22},...,a_{J2})^\top \in\R^{3J-1}:\delta_j >0, ~j=1,...,J  \right\}.
$$ If the true parameters
satisfy $a^*_{j2}=0, ~j=2,...,J,$ then $I(\ttt^*)$ is non-invertible.


(2) For the two-factor IFA model given in Example 2(a),
choose the parameter space to be $
\TT  = \R^{3J-1}.
$ If the true parameters
satisfy $a^*_{j2}=0, ~j=2,...,J,$ then $I(\ttt^*)$ is non-invertible.

\end{lemma}

We remark on the consequences of having a non-invertible information matrix. The first consequence is computational. If the information matrix is non-invertible, then the likelihood function does not tend to be strongly convex near the MLE, resulting in slow convergence.
In the context of Examples 1(a) and 2(a), it means that computing the MLE for the corresponding two-factor models may have convergence issue.
When convergence issue occurs, the obtained LRT statistic is below its actual value,
due to the log-likelihood for the two-factor model not achieving the maximum. Consequently, the $p$-value tends to be larger than its actual value, and thus the decision based on the $p$-value tends to be more conservative than the one without convergence issue.
This convergence issue is observed when conducting simulations for these examples. To improve the convergence, we use multiple random starting points when computing MLEs. The second consequence is a poor asymptotic convergence rate for the MLE. That is, the convergence rate is typically much slower than the standard parametric rate $N^{-1/2}$, even though the MLE is still consistent; see \cite{rotnitzky2000likelihood} for more theoretical results on this topic.

We further provide some remarks on the LRT in Examples 1(b) and 2(b) that use a LRT for comparing the fitted model with the saturated model. Although Wilks' theorem holds asymptotically in example 2(b), the $\chi^2$  approximation may not always work well as in our simulated example. This is because, when the number of items becomes larger and the sample size is not large enough, the contingency table for all $2^J$ response patterns may be sparse and thus the saturated model cannot be accurately estimated. In that case, it is better to use a limited-information inference method  \citep[e.g.][]{maydeu2005limited,maydeu2006limited} as a goodness-of-fit test statistic.
Similar issues might also occur to Example 1(b).



\section{General Theory for Likelihood Ratio Test}\label{sec:theory}

The previous discussions suggest that  Wilks' theorem does not hold for Examples 1(a), 2(a), and 3, due to the violation of regularity conditions. It is then natural to ask: what asymptotic distribution does $\lambda_N$ follow in these situations? Is there asymptotic theory characterizing such irregular situations? The answer to these questions is ``yes". In fact, a general theory characterizing these less regular situations has already been established in \cite{chernoff1954distribution}. In what follows, we provide a version of this general theory that is proven in \cite{van2000asymptotic}, Theorem 16.7. It is also given in \cite{drton2009likelihood}, Theorem 2.6. Two problems will be considered, (1) comparing a submodel with the saturated model as in Examples 1(b) and 2(b), and (2) comparing two submodels as in Examples 1(a), 2(a), and 3.

\subsection{Testing Submodel against Saturated Model}

We first introduce a few notations.
We use $\R^{J\times J}_{pd}$ and $\R^{J\times J}_d$ to denote the spaces of  $J\times J$   strictly positive definite matrices and diagonal matrices, respectively.
In addition, we define a one-to-one mapping $\rho$: $\mathcal \R_{pd}^{J\times J} \mapsto \mathbb R^{J(J+1)/2}$, that maps a positive definite matrix to a vector containing all its
upper triangular entries (including the diagonal entries). That is,
$\rho(\Sigma) = (\sigma_{11}, \sigma_{12}..., \sigma_{1J}, \sigma_{22}, ..., \sigma_{2J}, ..., \sigma_{JJ})^\top$,
for $\Sigma = (\sigma_{ij})_{J\times J} \in \mathcal \R_{pd}^{J\times J}$.
We also define a one-to-one mapping $\mu$: $\R_{d}^{J\times J} \mapsto \mathbb R^{J}$, that maps a diagonal matrix to a vector containing all its diagonal entries.

We consider to compare a submodel versus the saturated model. Let $\Theta_0$ and $\Theta$ be the parameter spaces of the submodel and the saturated model, respectively, satisfying $\Theta_0 \subset \Theta \subset \mathbb R^k$. Also let
$\ttt^* \in \Theta_0$ be the true parameter vector. The asymptotic theory of the LRT for comparing $\Theta_0$ versus
$\Theta$ requires regularity conditions C1-C5 below.


\begin{enumerate}
  \item[C1.] The true parameter vector $\ttt^*$ is in the interior of $\TT$.
  \item[C2.] There exists a measurable map $\dot l_{\ttt}:\R^J \to \R^k$ such that
  \begin{equation}\label{eq:quad}
  \lim_{\hh\to\bf0} \frac{1}{\|\hh\|^2} \int_{\R^J} \left( \sqrt{p_{\ttt+\hh}(\xx)} - \sqrt{p_{\ttt}(\xx)} - \frac{1}{2}\hh^\top\dot l_{\ttt}(\xx)\sqrt{p_{\ttt}(\xx)} \right)^2 d\nu(\xx) = 0,
  \end{equation} and the Fisher-information matrix $I(\boldsymbol\theta^*)$ for $\PPP_{\TT}$ is invertible.
  \item[C3.] There exists a neighborhood of $\ttt^*$, $U_{\ttt^*} \subset \TT$,  and a measurable function $\dot l: \R^J \to \R$, square integrable as $\int_{\R^J}\dot l(\xx)^2 dP_{\ttt^*}(\xx) < \infty,$  such that $$
  \vert  \log p_{\ttt_1}(\xx) - \log p_{\ttt_2}(\xx)  \vert \leq \dot l(\xx) \| \ttt_1 - \ttt_2 \|, \quad \forall \ttt_1,\ttt_2 \in U_{\ttt^*}.
  $$
  \item[C4.] The maximum likelihood estimators (MLE) $$\hat \ttt_{N, \TT} = \argmax_{\ttt\in \Theta} l_N(\ttt)$$
  and $$\hat \ttt_{N, \TT_0} = \argmax_{\ttt\in \Theta_0} l_N(\ttt)$$ are consistent under $P_{\ttt^*}.$
\end{enumerate}

The asymptotic distribution of $\lambda_N$ depends on the local geometry of the parameter space $\Theta_0$ at $\ttt^*$. This is characterized by the tangent cone $T_{\Theta_0}(\ttt^*)$, to be defined below.

\begin{definition}\label{def:1}
The tangent cone $T_{\TT_0}(\ttt^*)$ of the set $\TT_0 \subset \R^k$ at the point $\ttt^* \in \R^k$ is the set of vectors in $\R^k$ that are limits of sequences $\alpha_n(\ttt_n - \ttt^*),$ where $\alpha_n$ are positive reals and $\ttt_n \in \TT_0$ converge to $\ttt^*$.
\end{definition}
The following regularity is required for the tangent cone $T_{\Theta_0}(\ttt^*)$ that is known as the Chernoff-regularity.

\begin{enumerate}
  \item[C5.] For every vector $\boldsymbol \tau$ in the tangent cone $T_{\TT_0}(\ttt^*)$ there exist $\epsilon>0$ and a map $\boldsymbol\alpha:[0,\epsilon)\to \TT_0$ with $\boldsymbol\alpha(0) = \ttt^*$ such that $\boldsymbol \tau = \lim_{t\to 0+}[\boldsymbol\alpha(t) - \boldsymbol\alpha(0)]/t.$
\end{enumerate}


Under the above regularity conditions, Theorem~\ref{thm:LRT} below holds and explains the phenomena in Examples 1(b) and 2(b).

\begin{theorem}\label{thm:LRT}

  Suppose that conditions C1-C5 are satisfied for comparing nested models $\Theta_0 \subset \Theta \subset \mathbb R^k$, with $\ttt^* \in \Theta_0$ being the true parameter vector.
Then as $N$ grows to infinity,
the likelihood ratio statistic $\lambda_N$ converges to the distribution of
\begin{align}\label{eq:dist_general}
\min_{\boldsymbol\tau \in T_{\TT_0}(\ttt^*)} \| \mathbf Z - I(\ttt^*)^{\frac{1}{2}}\boldsymbol\tau \|^2,
\end{align}
where $\mathbf Z = (Z_1, ..., Z_k)^\top$ is a random vector consisting of i.i.d. standard normal random variables.

\end{theorem}


\begin{remark}
We give some remarks on the regularity conditions.
Conditions C1-C4 together ensure the asymptotic normality for $\sqrt{N}(\hat\ttt_{N,\TT}-\ttt^*)$.
Condition C1 depends on both the true model and the saturated model. As will be shown below, this condition holds for the saturated models in Examples 1(b) and 2(b).
Equation \eqref{eq:quad} in C2 is also known as the condition of ``differentiable in quadratic mean'' for $\PPP_{\TT}$ at $\ttt^*.$
If the map $\ttt \mapsto \sqrt{p_{\ttt}(\xx)}$ is continuously differentiable for every $\xx,$ then C2 holds with $\dot l_{\ttt}(\xx) = \frac{\partial}{\partial \ttt}\log p_{\ttt}(\xx)$ (Lemma 7.6, \cite{van2000asymptotic}). Furthermore, C3 holds if $\dot l(\xx) = \sup_{\ttt\in U_{\ttt^*}}\dot l_{\ttt}(\xx)$  is square integrable with respect to the measure $P_{\ttt^*}.$ Specifically, if $\dot l(\xx)$ is a bounded function, then C3 holds.
C4 holds for our examples by Theorem 10.1.6, \cite{casella2002statistical}. C5 requires certain regularity on the local geometry of $T_{\TT_0}(\ttt^*),$ which also holds for our examples below.
\end{remark}

\begin{remark}
By Theorem \ref{thm:LRT}, the asymptotic distribution for $\lambda_N$ depends on the tangent cone $T_{\TT_0}(\ttt^*).$ If $T_{\TT_0}(\ttt^*)$ is a linear subspace of $\R^k$ with dimension $k_0$, then one can easily show that the asymptotic reference distribution of $\lambda_N$ is
$\chi^2$  with degrees of freedom $k-k_0$. 
As we explain below,  Theorem~\ref{thm:LRT} directly applies to Examples \ref{eg1}(b) and \ref{eg2}(b).
If $T_{\TT_0}(\ttt^*)$ is a convex cone, then $\lambda_N$ converges to a mixture of $\chi^2$ distribution \citep{shapiro1985asymptotic,self.liang:1987}. That is, for any $x > 0$,  $P(\lambda_N \leq x)$ converges to $\sum_{i=0}^k w_k P(\xi_i  \leq x)$, as $N$ goes to infinity, where $\xi_0 \equiv 0$ and $\xi_i$ follows a $\chi^2$-distribution with $i$ degrees of freedom for $i > 0$. Moreover, the weights  sum up to 1/2 for the components with even degrees of freedom, and so do the weights for the components with odd degrees of freedom \citep{shapiro1985asymptotic}.


%
%
%
%
\end{remark}


 \begin{example}{\bf (Exploratory factor analysis, revisited).}
Now we consider Example~\ref{eg1}(b). As the saturated model is a $J$-variate normal distribution with an unrestricted covariance matrix, its parameter space can be chosen as




$$
\TT = \{ \rho(\boldsymbol\Sigma) : \boldsymbol\Sigma \in \R_{pd}^{J\times J} \} \subset \R^{J(J+1)/2},
$$ and the parameter space for the restricted model is $$
\TT_0 = \left\{ \rho(\boldsymbol\Sigma): \boldsymbol\Sigma =  \aaaa_1\aaaa_1^\top + \boldsymbol\Delta,~ \aaaa_1 \in \R^J, \boldsymbol\Delta \in \R_{pd}^{J\times J} \cap \R_{d}^{J\times J} \right\}.
$$ Suppose $\ttt^* = \rho(\boldsymbol\Sigma^*) \in \TT_0,$ where $\boldsymbol\Sigma^* = {\aaaa_1^*}{\aaaa_1^*}^\top + \boldsymbol\Delta^*.$ It is easy to see that C1 holds with the current choice of $\TT.$
The tangent cone $T_{\TT_0}(\ttt^*)$ takes the form:
$$
T_{\TT_0}(\ttt^*) = \left\{\rho(\boldsymbol\Sigma): \boldsymbol\Sigma = {\aaaa_1^*}\bb_1^\top + \bb_1{\aaaa_1^*}^\top + \mathbf B, ~ \bb_1 \in \R^{J}, \mathbf B \in \R_{d}^{J\times J}  \right\},
$$ which is a linear subspace of $\R^{J(J+1)/2}$ with dimension $2J,$ as long as $a^*_{j1} \neq 0, ~j=1,...,J.$
By Theorem \ref{thm:LRT}, $\lambda_N$ converges to the $\chi^2$-distribution with degrees of freedom $J(J+1)/2 - 2J.$

\end{example}

\begin{example}{\bf (Exploratory item factor analysis, revisited).}
Now we consider Example~\ref{eg2}(b). As the saturated model is a $2^J$-dimensional categorical distribution, its parameter space can be chosen as $$
\TT = \left\{  \ttt = \{\theta_{\xx}\}_{\xx\in\Gamma_J}: \theta_{\xx} \geq 0, \sum_{\xx\in\Gamma_J}\theta_{\xx} \leq 1
\right\} \subset \R^{2^J - 1},
$$ where $\Gamma_J := \{0,1\}^J \backslash \{(0,...,0)^\top\}.$
Then, the parameter space for the restricted model is
\begin{equation}
\begin{aligned}
\TT_0 = \left\{ \ttt \in \TT: \theta_{\xx} = \int \prod_{j=1}^J \frac{\exp(x_{j}(d_j + a_{j1}\xi_1))}{1+\exp(d_j + a_{j1}\xi_1)} \phi(\xi_1)d\xi_1, ~ \aaaa_1,\dd \in \R^J \right\}.
\end{aligned}
\end{equation}
Let $\ttt^* \in \TT_0$ that corresponds to true item parameters $\aaaa_1^* = (a^*_{j1},...,a^*_{J1})^\top$ and $\dd^* = (d^*_{1},...,d^*_{J})^\top.$
By the form of $\TT_0,$ $\ttt^*$ is an interior point of $\TT.$
For any $\xx \in \Gamma_J,$ we define $\ff_\xx = (f_{1}(\xx),...,f_{J}(\xx))^\top$ and $\gggg_{\xx} = (g_{1}(\xx),...,g_{J}(\xx))^\top,$ where $$
f_{l}(\xx) = \int \prod_{j=1}^J \frac{\exp(x_{j}(d^*_{j} + a^*_{j1}\xi_1))}{1+\exp(d^*_{j} + a^*_{j1}\xi_1)} \left[ x_l - \frac{\exp(d^*_{l} + a^*_{l1}\xi_1)}{1+\exp(d^*_{l} + a^*_{l1}\xi_1)} \right] \phi(\xi_1) d\xi_1,
$$ and $$
g_{l}(\xx) = \int \prod_{j=1}^J \frac{\exp(x_{j}(d^*_{j} + a^*_{j1}\xi_1))}{1+\exp(d^*_{j} + a^*_{j1}\xi_1)} \left[ x_l - \frac{\exp(d^*_{l} + a^*_{l1}\xi_1)}{1+\exp(d^*_{l} + a^*_{l1}\xi_1)} \right] \xi_1\phi(\xi_1) d\xi_1,
$$ for $l=1,...,J.$
Then the tangent cone $T_{\TT_0}(\ttt^*)$ has the form $$
T_{\TT_0}(\ttt^*) = \left\{\ttt = \{\theta_{\xx}\}_{\xx\in \Gamma_J}: \theta_{\xx} = \bb_0^\top \ff_{\xx} + \bb_1^\top\gggg_{\xx},~\bb_0,\bb_1\in \R^J \right\},
$$
which is a linear subspace of $\R^{2^J-1}$ with dimension $2J.$
By Theorem \ref{thm:LRT}, $\lambda_N$ converges to the distribution of $\chi^2$  with degrees of freedom $2^{J}-1 - 2J.$

\end{example}

\subsection{Comparing Two Nested Submodels}

Theorem~\ref{thm:LRT} is not applicable to Example 3, because $\ttt^*$ is on the boundary of $\TT$ if $\TT$ is chosen to be $\{(\beta_0, \sigma_1^2, \sigma_2^2): \beta_0 \in \mathbb R, \sigma_1^2 \in [0, \infty), \sigma_2^2 \in [0, \infty)\},$ and thus C1 is violated. Theorem~\ref{thm:LRT} is also not applicable to Examples 1(a) and 2(a), because the Fisher information matrix is not invertible when $\Theta$ is chosen to be the parameter space of the two-factor EFA and IFA models, respectively, in which case
condition C2 is violated.

To derive the asymptotic theory for such problems, we view them as a problem of  testing nested submodels under a saturated model for which $\ttt^*$ is an interior point of $\TT$ and the information matrix is invertible. Consider testing
$$H_0: \ttt \in \Theta_0 \mbox{~versus~} H_a:  \ttt \in \Theta_1\setminus \Theta_0,$$
where $\Theta_0$ and $\Theta_1$ are two nested submodels of a saturated model $\Theta$, satisfying $\Theta_0 \subset \Theta_1 \subset \Theta \subset \mathbb R^{k}$. Under this formulation, Theorem~\ref{thm:LRT2} below provides the asymptotic  theory for the LRT statistic $\lambda_N = 2\left(\sup_{\ttt\in \Theta_1} l_N(\ttt) - \sup_{\ttt\in \Theta_0} l_N(\ttt)\right)$.

To obtain the asymptotic distribution of $\lambda_N$, regularity conditions C1-C5 are still required for $\Theta_0 \subset \Theta$. Two additional conditions are needed for $\Theta_1$, which are satisfied for Examples 6, 7 and 8 below.

\begin{enumerate}
\item[C6.] The MLE under $\Theta_1$,
$\hat \ttt_{N, \TT_1} = \argmax_{\ttt\in \Theta_1} l_N(\ttt)$, is consistent under $P_{\ttt^*}.$
  \item[C7.] Let $T_{\TT_1}(\ttt^*)$ be the  tangent cone for $\Theta_1$, defined the same as in Definition~\ref{def:1} but with $\Theta_0$ replaced by $\Theta_1$. $T_{\TT_1}(\ttt^*)$ satisfies Chernoff regularity. That is, for every vector $\boldsymbol \tau$ in the tangent cone $T_{\TT_1}(\ttt^*)$ there exist $\epsilon>0$ and a map $\boldsymbol\alpha:[0,\epsilon)\to \TT_1$ with $\boldsymbol\alpha(0) = \ttt^*$ such that $\boldsymbol \tau = \lim_{t\to 0+}[\boldsymbol\alpha(t) - \boldsymbol\alpha(0)]/t.$
\end{enumerate}




\begin{theorem}\label{thm:LRT2}

Let $\ttt^* \in \TT_0$ be the true parameter vector. Suppose that conditions C1-C7 are satisfied.
As $N$ grows to infinity,
the likelihood ratio statistic $\lambda_N$ converges to the distribution of
\begin{equation}\label{eq:dist_nested}
\begin{aligned}
\min_{\boldsymbol\tau \in T_{\TT_0}(\ttt^*)} \| \mathbf Z - I(\ttt^*)^{\frac{1}{2}}\boldsymbol\tau \|^2 -
\min_{\boldsymbol\tau \in T_{\TT_1}(\ttt^*)} \| \mathbf Z - I(\ttt^*)^{\frac{1}{2}}\boldsymbol\tau \|^2,
\end{aligned}
\end{equation}
where $\mathbf Z = (Z_1, ..., Z_k)^\top$ is a random vector consisting of i.i.d. standard normal random variables, and
$I(\ttt^*)^{\frac{1}{2}}$ satisfies $I(\ttt^*)^{\frac{1}{2}} (I(\ttt^*)^{\frac{1}{2}})^\top = I(\ttt^*)$ that can be obtained by eigenvalue decomposition.

\end{theorem}


\begin{example}{\bf (Random effects model, revisited).}
Now we consider Example~\ref{eg3}.
Let $\EE_n$ denote a length-$n$ vector whose entries are all 1, and $\II_n$ denote the $n\times n$ identity matrix.
As $\XX_i = (X_{i1},...,X_{iJ})^\top$ from the random effects model is multivariate normal with mean $\beta_0\EE_J$ and covariance matrix $\sigma_1^2\EE_J\EE_J^\top+ \sigma_2^2\II_J,$ the saturated parameter space can be taken as
$$
\TT = \{ (\rho(\boldsymbol\Sigma)^\top,\beta_0)^\top: \boldsymbol\Sigma \in \R^{J\times J}_{pd}, \beta_0\in \R \}.
$$ The parameter space for restricted models are $$
\TT_0 = \{ (\rho(\boldsymbol\Sigma)^\top,\beta_0)^\top:\boldsymbol\Sigma = \sigma_2^2\II_J,~ \sigma_2^2 >0, \beta_0\in \R \},
$$ and $$
\TT_1 = \{ (\rho(\boldsymbol\Sigma)^\top,\beta_0)^\top:\boldsymbol\Sigma = \sigma_1^2\EE_J\EE_J^\top + \sigma_2^2\II_J, \sigma_1^2\geq0, \sigma_2^2>0, \beta_0\in \R \}.
$$ Let $\ttt^* = (\rho(\boldsymbol\Sigma^*),\beta^*_0) \in \TT_0,$ where $\boldsymbol\Sigma^* = {\sigma_2^*}^2\II_J.$ 
Then, C1 holds.
The tangent cones for $\TT_0$ and $\TT_1$ are $$
T_{\TT_0}(\ttt^*) = \{ (\rho(\boldsymbol\Sigma)^\top,b_0)^\top:\boldsymbol\Sigma = b_2\II_J,~ b_0,b_2\in \R \}
$$ and $$
T_{\TT_1}(\ttt^*) = \{ (\rho(\boldsymbol\Sigma)^\top,b_0)^\top:\boldsymbol\Sigma = b_1\EE_J\EE_J^\top + b_2\II_J,~ b_1\geq0, b_0,b_2\in \R \}.
$$
By Theorem \ref{thm:LRT2}, $\lambda_N$ converges to the distribution of \eqref{eq:dist_nested}.

In this example, the form of \eqref{eq:dist_nested} can be simplified, thanks to the forms of $T_{\TT_0}(\ttt^*)$ and $T_{\TT_1}(\ttt^*).$
We denote $$
\cc_0 = (0,...,0,1), ~\cc_1 = (\rho(\EE_J\EE_J^\top)^\top,0)^\top,~\cc_2 = (\rho(\II_J)^\top,0)^\top \in \R^{J(J+1)/2+1}.
$$ It can be seen that $T_{\TT_0}(\ttt^*)$ is a 2-dimensional linear subspace spanned by $\{\cc_0,\cc_2\},$ and $T_{\TT_1}(\ttt^*)$ is a half 3-dimensional linear subspace defined as $\{\alpha_0\cc_0+\alpha_1\cc_1+\alpha_2\cc_2: \alpha_1\geq0,\alpha_0,\alpha_2\in \R\}.$
Let $\PP_0$ denote the projection onto $T_{\TT_0}(\ttt^*).$ Define $$
\vv = \frac{\cc_1 - \PP_0\cc_1}{\|\cc_1 - \PP_0\cc_1\|},
$$ and then \eqref{eq:dist_nested} has the form
\begin{equation}\label{eq:mixchi}
\|\vv^\top \ZZ\|^21_{\{\vv^\top \ZZ \geq 0\}}.
\end{equation}
It is easy to see that $\vv^\top\ZZ$ follows standard normal distribution.
Therefore, $\lambda_N$ converges to the distribution of $w^21_{\{w\geq0\}},$ where $w$ is a standard normal random variable. This is known as a mixture of $\chi^2$-distribution.
The blue dotted line in Figure~\ref{fig:fig3} shows the CDF of this mixture $\chi^2$-distribution. This CDF is very close to the empirical CDF of the LRT, confirming our asymptotic theory.



\end{example}

\begin{example}{\bf (Exploratory factor analysis, revisited).}
Now we consider Example~\ref{eg1}(a). Let $\TT,\TT_0,\ttt^*$ and $T_{\TT_0}(\ttt^*)$ be the same as those in Example 4.
In addition, we define $$
\TT_1 = \left\{ \rho(\boldsymbol\Sigma): \boldsymbol\Sigma =  \aaaa_1\aaaa_1^\top + \aaaa_2\aaaa_2^\top + \boldsymbol\Delta,~ \aaaa_1,\aaaa_2 \in \R^J, a_{12}=0, \boldsymbol\Delta \in \R_{pd}^{J\times J} \cap \R_{d}^{J\times J} \right\}.
$$ The tangent cone of $\TT_1$ at $\ttt^*$ becomes
$$
T_{\TT_1}(\ttt^*) = \left\{\rho(\boldsymbol\Sigma): \boldsymbol\Sigma = {\aaaa_1^*}\bb_1^\top + \bb_1{\aaaa_1^*}^\top + \bb_2\bb_2^\top + \mathbf B, ~ \bb_1,\bb_2 \in \R^{J}, b_{12}=0, \mathbf B \in \R_{d}^{J\times J}  \right\}.
$$
Note that $T_{\TT_1}(\ttt^*)$ is not a linear subspace, due to the $\bb_2\bb_2^\top$ term. Therefore, by Theorem \ref{thm:LRT2}, 
the asymptotic distribution of $\lambda_N$ is not $\chi^2$ .
See the blue dotted line in Panel (a) of Figure~\ref{fig:fig1} for the CDF of this asymptotic distribution. This CDF almost overlaps with the empirical CDF of the LRT, suggesting that Theorem~\ref{thm:LRT2} holds here.

\end{example}

\begin{example}{\bf (Exploratory item factor analysis, revisited).}
Now we consider Example~\ref{eg2}(a). Let $\TT,\TT_0,\ttt^*$ and $T_{\TT_0}(\ttt^*)$ be the same as those in Example 5. Let $$
\begin{aligned}
\TT_1 = \left\{ \ttt \in \TT: \theta_{\xx} = \int\int \prod_{j=1}^J \frac{\exp(x_{j}(d_j + a_{j1}\xi_1+a_{j2}\xi_2))}{1+\exp(d_j + a_{j1}\xi_1+a_{j2}\xi_2)} \phi(\xi_1)\phi(\xi_2)d\xi_1d\xi_2, a_{12} = 0,\xx \in\Gamma_J \right\}
\end{aligned}
$$ be the parameter space for the two-factor model.
Recall $\ff_{\xx}$ and $\gggg_{\xx}$ as defined in Example 5. For any $\xx \in\Gamma_J,$
we further define $\HH_{\xx} = (h_{rs}(\xx))_{J\times J},$ where $$
\begin{aligned}
h_{rs}(\xx) =&
\int \prod_{j=1}^J \frac{\exp(x_{j}(d^*_{j} + a^*_{j1}\xi_1))}{1+\exp(d^*_{j} + a^*_{j1}\xi_1)} \left[ x_r - \frac{\exp(d^*_{r} + a^*_{r1}\xi_1)}{1+\exp(d^*_{r} + a^*_{r1}\xi_1)} \right] \\
& \times \left[ x_s - \frac{\exp(d^*_{s} + a^*_{s1}\xi_1)}{1+\exp(d^*_{s} + a^*_{s1}\xi_1)} \right] \phi(\xi_1)d\xi_1
\end{aligned}
$$ for $r\neq s,$ and
\begin{equation*}
\begin{aligned}
h_{rr}(\xx) =& \int \prod_{j=1}^J \frac{\exp(x_{j}(d^*_{j} + a^*_{j1}\xi_1))}{1+\exp(d^*_{j} + a^*_{j1}\xi_1)} \left\{ \left[ x_r - \frac{\exp(d^*_{r} + a^*_{r1}\xi_1)}{1+\exp(d^*_{r} + a^*_{r1}\xi_1)} \right]^2 \right. \\
& \left.- \frac{\exp(d_r^* + a_{r1}^*\xi_1)}{(1+\exp(d_r^* + a_{r1}^*\xi_1))^2} \right\} \phi(\xi_1)d\xi_1.
\end{aligned}
\end{equation*}
Then, the tangent cone of $\TT_1$ at $\ttt^*$ is
\begin{equation}
T_{\TT_1}(\ttt^*) =  \left\{ \ttt=\{\theta_{\xx}\}_{\xx\in\Gamma_J}: \theta_{\xx} =
\bb_0^\top \ff_{\xx} + \bb_1^\top \gggg_{\xx} + \bb_2^\top \HH_{\xx}\bb_2,~ \bb_0,\bb_1,\bb_2 \in \R^J, b_{12} = 0
\right\}.
\end{equation}
Similar to Example 7, $T_{\TT_1}(\ttt^*)$ is not a linear subspace and thus $\lambda_N$ is not asymptotically $\chi^2$ .
In Panel (a) of Figure~\ref{fig:fig2}, the asymptotic CDF suggested by Theorem \ref{thm:LRT2} is shown as the blue dotted line. Similar to the previously examples, this CDF is very close to the empirical CDF of the LRT.


\end{example}

\section{Discussion}\label{sec:conc}

In this note, we point out how the regularity conditions of Wilks' theorem may be violated, using three examples of models with latent variables.
In these cases, the asymptotic distribution of the LRT statistic is no longer  $\chi^2$  and therefore the test may no longer be valid. It seems that the regularity conditions of Wilks' theorem, especially the requirement on a non-singular
Fisher information matrix, have not received enough attention. As a result, the LRT is often misused. Although we focus on LRT, it is  worth pointing out that other testing procedures, including the Wald and score tests, as well as limited-information tests (e.g., tests based on bivariate information),
require similar regularity conditions and thus may also be affected.

We present a general theory for LRT first established in \cite{chernoff1954distribution} that is not widely known in psychometrics and related fields. As we illustrate by the three examples, this theory applies to irregular cases not covered by Wilks' theorem. There are other examples for which this general theory is useful. For example, Examples 1(a) and 2(a) can be easily generalized to the comparison of factor models with different numbers of factors, under both confirmatory and exploratory settings. This theory can also be applied to model comparison in latent class analysis that also suffers from a non-invertible information matrix.
To apply the theorem, the key is to choose a suitable parameter space and then
characterize the tangent cone at the true model.

There are alternative inference methods for making statistical inference under such irregular situations.
One method is to obtain a reference distribution for LRT via parametric bootstrap. Under the same regularity conditions as in Theorem~\ref{thm:LRT2}, we believe that the parametric bootstrap is still consistent.
The parametric bootstrap may even achieve better approximation accuracy for finite sample data than the asymptotic distributions given by Theorems \ref{thm:LRT} and \ref{thm:LRT2}. However, for complex latent variable models (e.g., IFA models with many factors), the parametric bootstrap may be computationally intensive, due to the high computational cost of repeatedly computing the marginal maximum likelihood estimators.
On the other hand, Monte Carlo simulation of the asymptotic distribution in Theorem~\ref{thm:LRT2} is computationally much easier, even though there are still optimizations to be solved. Another method is the split likelihood ratio test recently proposed by \cite{wasserman2020universal} that is computationally fast and does not suffer from singularity or boundary issues. By making use of a sample splitting trick, this split LRT is able to control the type I error at any pre-specified level. However, it may  be quite conservative sometimes.

This paper focuses on the situations where the true model is exactly a singular or boundary point
of the parameter space. However, the LRT can also be problematic when the true model is near a singular or boundary point. A recent article by \cite{mitchelletal:2019}  provides a treatment of this problem, where a finite sample approximating distribution is derived for LRT.

Besides the singularity and boundary issues, the asymptotic distribution may be inaccurate when the dimension of the parameter space is relatively high comparing with the sample size. This problem has been intensively studied in statistics and
a famous result is the Bartlett correction which provides a way to improve the $\chi^2$  approximation \citep{bartlett1937properties,bickel1990decomposition,cordeiro1983improved,box1949general,lawley1956general,wald1943tests}.
When the regularity conditions do not hold, the classical form of Bartlett correction may no longer be suitable. A general form of Bartlett correction remains to be developed, which is left for future investigation.

\bigskip\bigskip\bigskip

\newenvironment{proof}[1][Proof]{\noindent\textbf{#1.} }{\ \rule{0.5em}{0.5em}}

\appendix
\noindent
{\Large Appendix}

\begin{proof}[Proof of Lemma \ref{lemma:lemma1}]
Denote the $(i,j)$-entry of the Fisher-information matrix $I(\ttt^*)$ as $q_{ij}.$
In both cases, we show that $q_{ij} = 0$ for $i \geq 2J+1,$ or $j \geq 2J+1,$ and therefore $I({\ttt^*})$ is non-invertible. Since $$
q_{ij} = \int \frac{\partial \log p_{\ttt}\left(\x\right) }{\partial \theta_i} \Bigr|_{\ttt^*}\frac{\partial \log p_{\ttt}\left(\x\right) }{\partial \theta_j}\Bigr|_{\ttt^*} p_{\ttt^*}(\x)d\x,
$$ it suffices to show that $$
\frac{\partial \log p_{\ttt}\left(\x\right) }{\partial \theta_i}\Bigr|_{\ttt^*} = 0, \quad j \geq 2J+1.
$$

In the case of two-factor model, it suffices to show that $$
\frac{\partial \log p_{\ttt}\left(\x\right) }{\partial a_{l2}} \Bigr|_{\ttt^*} = 0,
$$ for $l=2,...,J.$ Let $\sigma_{ij}$ be the $(i,j)$-entry of the covariance matrix $\Sigma$ and it is easy to see that $\sigma_{ij} = a_{i1}a_{j1}+a_{i2}a_{j2}+1_{\{i=j\}}\delta_i,$ where $a_{12} = 0.$ By the chain rule, $$
\frac{\partial \log p_{\ttt}\left(\x\right) }{\partial a_{l2}} = \sum_{i\leq j} \frac{\partial \log p_{\ttt}\left(\x\right) }{\partial \sigma_{ij}} \frac{\partial \sigma_{ij}}{\partial a_{l2}}.
$$ Since $$\begin{aligned}
\frac{\partial \sigma_{ij}}{\partial a_{l2}} \Bigr|_{\ttt^*} &= 1_{\{l=i\}}a^*_{j2} + 1_{\{l=j\}}a^*_{i2}\\
&= 0,
\end{aligned}
$$ then $I(\ttt^*)$ is non-invertible in the case of two-factor model.

In the case of two-factor IFA model, since $$
\frac{\partial \log p_{\ttt}\left(\x\right) }{\partial \theta_i} = \frac{1}{p_{\ttt}(\x)}\frac{\partial p_{\ttt}\left(\x\right) }{\partial a_{l2}}
$$ it suffices to show that $$
\frac{\partial p_{\ttt}\left(\x\right) }{\partial a_{l2}} \Bigr|_{\ttt^*} = 0,
$$ for $l=2,...,J.$ Since $$\begin{aligned}
\frac{\partial p_{\ttt}\left(\x\right) }{\partial a_{l2}} \Bigr|_{\ttt^*} &= \int\int \prod_{j=1}^J \frac{\exp(x_{j}(d^*_{j} + a^*_{j1}\xi_1))}{1+\exp(d^*_{j} + a^*_{j1}\xi_1)} \left[ x_l - \frac{\exp(d^*_{l} + a^*_{l1}\xi_1)}{1+\exp(d^*_{l} + a^*_{l1}\xi_1)} \right] \xi_2\phi(\xi_1)\phi(\xi_2) d\xi_1d\xi_2\\
&= \int \xi_2 \phi(\xi_2) d\xi_2 \times
\int \prod_{j=1}^J \frac{\exp(x_{j}(d^*_{j} + a^*_{j1}\xi_1))}{1+\exp(d^*_{j} + a^*_{j1}\xi_1)} \left[ x_l - \frac{\exp(d^*_{l} + a^*_{l1}\xi_1)}{1+\exp(d^*_{l} + a^*_{l1}\xi_1)} \right] \phi(\xi_1) d\xi_1 \\
&=0,
\end{aligned}
$$ then $I(\ttt^*)$ is non-invertible in the case of two-factor IFA model.
\end{proof}

\begin{proof}[Proof of Theorem \ref{thm:LRT}]
We refer readers to \cite{drton2009likelihood}, Theorem 2.6.
\end{proof}

\begin{proof}[Proof of Theorem \ref{thm:LRT2}]
The proof is similar to that of Theorem 16.7, \cite{van2000asymptotic}.
We only state the main steps and skip the details which readers can find in \cite{van2000asymptotic}.

We introduce some notations. Let $$T_{N,0} = \left\{ \sqrt{N}(\ttt-\ttt^*): \ttt\in \TT_0 \right\}$$ and $$T_{N,1} = \left\{ \sqrt{N}(\ttt-\ttt^*): \ttt\in \TT_1 \right\}.
$$ Under conditions C5 and C7, $T_{N,0}, T_{N,1}$ converge to $T_{\TT_0}(\ttt^*)$ and $T_{\TT_1}(\ttt^*),$ respectively in the sense of \cite{van2000asymptotic}. Let $I(\ttt^*)^{-\frac{1}{2}}$ denote the inverse of $I(\ttt^*)^{\frac{1}{2}}.$
Let $\G_N = \sqrt{N}(\PPP_N - P_{\ttt^*})$ be the empirical process.
Then,
\begin{align}
\lambda_N & = 2\sup_{\ttt\in \TT_1}l_{N}(\ttt) - 2\sup_{\ttt\in \TT_0}l_{N}(\ttt) \nonumber\\
&= 2\sup_{\hh\in T_{N,1}} N\PPP_N \log p_{\ttt^*+\hh/\sqrt{N}}(\x) - 2\sup_{\hh\in T_{N,0}} N\PPP_N \log p_{\ttt^*+\hh/\sqrt{N}}(\x) \nonumber\\
&= 2\sup_{\hh\in T_{N,1}} N\PPP_N \log \frac{p_{\ttt^*+\hh/\sqrt{N}}(\x)}{p_{\ttt^*}(\x)} - 2\sup_{\hh\in T_{N,0}} N\PPP_N \log \frac{p_{\ttt^*+\hh/\sqrt{N}}(\x)}{p_{\ttt^*}(\x)}\nonumber\\
&= 2\sup_{\hh\in T_{N,1}} \left(\hh^\top\G_N\dot l_{\ttt^*} - \frac{1}{2}\hh^\top I(\ttt^*)\hh\right) - 2\sup_{\hh\in T_{N,0}} \left(\hh^\top\G_N\dot l_{\ttt^*} - \frac{1}{2}\hh^\top I(\ttt^*)\hh\right) + o_p(1) \label{eq:tmp1}\\
&= \sup_{\hh\in T_{\TT_0}(\ttt^*)}\left\| I(\ttt^*)^{-\frac{1}{2}}\G_N\dot l_{\ttt^*} - I(\ttt^*)^{\frac{1}{2}}\hh \right\|^2 - \sup_{\hh\in T_{\TT_1}(\ttt^*)}\left\| I(\ttt^*)^{-\frac{1}{2}}\G_N\dot l_{\ttt^*} - I(\ttt^*)^{\frac{1}{2}}\hh \right\|^2 + o_p(1)\label{eq:tmp2}.
\end{align}
The $\dot l_{\ttt^*}$ is defined by condition C2.
For details of \eqref{eq:tmp1}, see the proof of Theorem 16.7,   \cite{van2000asymptotic}. \eqref{eq:tmp2} is derived from $$
2\hh^\top\G_N\dot l_{\ttt^*} - \hh^\top I(\ttt^*)\hh = - \left\| I(\ttt^*)^{-\frac{1}{2}}\G_N\dot l_{\ttt^*} - I(\ttt^*)^{\frac{1}{2}}\hh \right\|^2 + \left\| I(\ttt^*)^{-\frac{1}{2}} \G_N\dot l_{\ttt^*}\right\|^2,
$$ and the fact that $T_{N,0}, T_{N,1}$ converge to $T_{\TT_0}(\ttt^*)$ and $T_{\TT_1}(\ttt^*),$ respectively.
By central limit theorem, $I(\ttt^*)^{-\frac{1}{2}}\G_N\dot l_{\ttt^*}$ converges to $k$-variate standard normal distribution. We complete the proof by continuous mapping theorem.
\end{proof}

\bibliographystyle{apacite}
\bibliography{ref}

\end{document}